\documentclass[a4paper,10pt]{amsproc}
\title{A topological view of algebraic structures in $C_p(X)$}
\usepackage{amsfonts, amsmath, amssymb, graphicx, mathrsfs, geometry, enumitem}

\usepackage[all]{xy}
\usepackage{xcolor}
\usepackage{hyperref}
\hypersetup{
	colorlinks=true,
	linkcolor={red!60!black},
	citecolor={blue!50!black},
	urlcolor={green!50!black}
}

\newdir{ >}{!/6pt/@{ }*:(1,0.2)@_{>}*:(1,-0.2)@^{>}}

\theoremstyle{plain}
\newtheorem{theorem}{Theorem}[section]
\newtheorem{lemma}[theorem]{Lemma}

\theoremstyle{definition}
\newtheorem{definition}[theorem]{Definition}
\theoremstyle{definition}
\newtheorem{definitions}[theorem]{Definitions}
\theoremstyle{definition}

\newtheorem{remark}[theorem]{Remark}
\newtheorem{counter example}[theorem]{Counter Example}

\newtheorem{notations}[theorem]{Notations}
\newtheorem{corollary}[theorem]{Corollary}

\newtheorem{question}[theorem]{Question}
\numberwithin{equation}{section}

\author[P. Nandi]{Pratip Nandi}	\address{Centro de Ciencias Matemáticas, Universidad Nacional Autónoma de México, Antigua Carretera a Pátzcuaro 8701, Col. Ex-Hacienda de San José de la Huerta, C.P. 58089 Morelia, Michoacán. México}	\email{pratipnandi10@gmail.com}

\author[S. Dey]{Soumajit Dey}	\address{Department of Pure Mathematics, University of Calcutta, 35, Ballygunge Circular Road, Kolkata 700019, West Bengal, India}	\email{deysoumajit8@gmail.com}

\author[A. Dey]{Amrita Dey}	\address{Department of Pure Mathematics, University of Calcutta, 35, Ballygunge Circular Road, Kolkata 700019, West Bengal, India}	\email{deyamrita0123@gmail.com}

\author[S. K. Acharyya]{Sudip Kumar Acharyya}	\address{Department of Pure Mathematics, University of Calcutta, 35, Ballygunge Circular Road, Kolkata 700019, West Bengal, India}	\email{sdpacharyya@gmail.com}

\subjclass[2020]{Primary ; Secondary }

\begin{document}
	
	\title{A topological view of algebraic structures in $C_p(X)$}

	%

	\thanks {The author is immensely grateful for the award of research fellowship provided by the University Grants Commission, New Delhi (NTA Ref. No. 221610014636).}

	\large
	\begin{abstract}
		This manuscript focuses on the topological behaviour of algebraic structures in $C_p(X)$. Closure of ideals in $C_p(X)$ is determined and are used to characterise various topological properties of the underlying Tychonoff space $X$. These include compact space, locally compact space, locally pseudocompact space, $P$-space, almost $P$-space and normal space. Moreover, it has been established that a space $X$ is totally separated space if and only if the set of all units in $C(X)$ is dense in $C_p(X)$. A subring of $C(X)$ is found to be dense in $C_p(X)$ when and only when it separates points. 
	\end{abstract}	
	
	\maketitle
	
	\section{Introduction}
	
	Vladimir V. Tkachuk wrote in his book `A $C_p$-Theory Problem Book - Topological and Function Spaces' \cite{Tkachuk} that ``The credit for the
	creation of $C_p$-theory must undoubtedly be given to Alexander Vladimirovich	Arhangel’skii." The \textit{topology of pointwise convergence} (or, the \textit{point-open topology}) is defined on the ring $C(X)$ of all real-valued continuous functions as follows: 
	
	
	\begin{definition}\cite[Problem 058]{Tkachuk}
		Let $X$ be a Tychonoff space. Also suppose $f\in C(X)$, $x_1,x_2,\cdots,x_n\in X$, where $n\in \mathbb{N}$ and $\epsilon$ is a positive real number. Then define $$B(f,x_1,x_2,\cdots,x_n,\epsilon):= \{g\in C(X)\colon |g(x_i)-f(x_i)|<\epsilon,\forall\;i\in \{1,2,\cdots,n\} \}.$$ The family $\{B(f,x_1,x_2,\cdots,x_n,\epsilon)\colon f\in C(X),\; x_1,x_2,\cdots,x_n\in X, \text{ where }n\in \mathbb{N}\;and\;\epsilon>0  \}$ forms a base for the topology of pointwise convergence on $C(X)$.
	\end{definition} 
	
	Although, the point-open topology on $C(X)$ has various equivalent definitions, the aforementioned definition is most widely used and shall be used throughout this article as well.
	
	We note that the assumption `$X$ is a Tychonoff space' is made in light of the M. H. Stone's Theorem \cite[Theorem 3.9]{GJ1976}. Furthermore, if $X$ is a singleton set, then the space $C_p(X)$ is precisely the real-line. So, it is assumed that $|X|\geq 2$. This shall be prevalent throughout this manuscript. The fact that $C_p(X)$ is a topological ring can be verified using routine arguments. This implies that the closure of ideals in $C_p(X)$ shall also be an ideal. We must clarify at this point that, whenever we use the term `ideal', we mean a proper ideal in $C(X)$. So, the closure of a maximal ideal shall either be closed or dense in $C_p(X)$. 
	
	Section 2 is devoted to the determination of closures of various ideals in $C_p(X)$. For a subset $A$ of $C(X)$, the closure of $A$ in $C_p(X)$ is denoted by $cl_p(A)$. Note that a maximal ideal in $C_p(X)$ is either closed or dense. It has been verified that a maximal ideal in $C_p(X)$ is closed if and only if it is fixed and is dense if and only if it is a free ideal \cite[Section 4.1]{GJ1976}. These revelations led to a characterisation of compact topological spaces using the space $C_p(X)$. Following this, the closure of $O^p$ in $C_p(X)$ is determined to be $M^p$, whenever $p\in X$ and $C(X)$, when $p\in \beta X\setminus X$. Hereafter, recall that a space $X$ is said to be \textit{pseudocompact} if all real-valued continuous functions on $X$ are bounded. The closure of $C_K(X)$ (resp.\ $C_\psi(X)$) in $C_p(X)$ is determined and is associated with the set containing all such points in $X$ which have a compact (resp.\ pseudocompact) neighbourhood. Using these, local compactness and local pseudocompactness of a topological space $X$ have been characterised in view of $C_p(X)$. In this context, we recall that a space $X$ is said to be \textit{locally pseudocompact} if each point in $X$ has a neighbourhood, on which every function in $C(X)$ is bounded. These results display a connection between the closed ideals in $C_p(X)$ and the fixed ideals in $C(X)$. This helped us observe that the closure of an ideal $I$ in $C_p(X)$ is given by \[cl_p(I)=\{f\in C(X)\colon \bigcap Z[I]\subseteq Z(f) \}. \] As a consequence, it is noted that $cl_p(O_p)=M_p$, for each $p\in X$ and this led to a characterisation of $P$-spaces (a space $X$ is said to be a \textit{$P$-space} if every $G_\delta$-set in $X$ is open \cite[Exercise 4J]{GJ1976}). Building from these discussions, the closure of prime ideals in $C_p(X)$ have been determined. Also, we recall that for $p\in X$, $O_p$ is a \textit{$z^\circ$-ideal}, i.e., $int\;Z(f)=int\;Z(g)$ with $f\in C(X)$ and $g\in O_p$ implies that $f\in O_p$ \cite{Levy}.  So, as a natural continuation, $z^\circ$-ideals are next studied. In this context, recall that a space is called \textit{almost $P$-space} if every non-empty $G_\delta$-set in $X$ has non-empty interior. Using the above mentioned discussions, almost $P$-spaces are characterised as spaces for which the closed ideals in $C_p(X)$ are $z^\circ$-ideals. An ideal $I$ in $C(X)$ is said to be an \textit{essential ideal} if $int\; \bigcap Z[I]=\emptyset$ \cite{A1995}.  Clearly, a free ideal in $C(X)$ is always essential, and hence is dense. It is realised that essential ideals in $C_p(X)$ are dense if and only if $X$ is discrete. Following this, we have established few conditions equivalent to each ideal being closed in $C_p(X)$. Next, it is established that the closure of any ideal $I$ in $C_p(X)$ can be written as the intersection of the closure of all maximal ideals containing it. The final discussion of this section is on the closure of sum of ideals in $C_p(X)$. Normal spaces have been characterised with the help of these discussions.
	
	In Section 3, we shift our focus from ideals in $C(X)$ to its subrings and various other class of algebraic structures. We begin with the collection of all divisors of zero in $C(X)$ along with $\boldsymbol{0}$, $D(X)$. We realise that this collection is either dense or closed in $C_p(X)$. Moreover, we characterise spaces for which $D(X)$ is dense (resp.\ closed) as spaces with infinite (resp.\ finite) cardinality. After this, we focus on the collection, $U(X)$ of all units in $C(X)$. We first realise that $U(X)$ is dense in $C_p(X)$ if and only if $X$ is a \textit{totally separated space}, i.e. each quasicomponent in $X$ is a singleton set. In this context, we note that if a space $X$ is either zero-dimensional or extremally disconnected, then $X$ is totally separated. Subsequently, an explicit description of the closure of $U(X)$ in $C_p(X)$ is provided. It has been realised that closure of the set $V(X)$ consisting of all Von-Neumann regular elements in $C(X)$ is the same as that of $U(X)$, for any space $X$. The rest of the article is dedicated to subrings of $C(X)$, in view of the topology of pointwise convergence. The focus is shifted to subrings of the form $C_\infty^\mathscr{P}(X)$ from \cite{ARN2022}. It is established that for each ideal $\mathscr{P}$ of closed sets in $X$, the closure of the subring $C_\infty^\mathscr{P}(X)$ is the same as that of $C_\mathscr{P}(X)$, which has been discussed in Section 2. This again leads to a characterisation of locally $\mathscr{P}$ (and hence, locally compact and locally pseudocompact) spaces \cite{ARN2022} using the subring $C_\infty^\mathscr{P}(X)$. Following this, we establish that a subring $S$ of $C(X)$ which contains all the constant functions is dense in $C_p(X)$ if and only if it separates points in $X$. 
	It follows from this that the ring $C^*(X)$ of all bounded functions in $C(X)$ is dense in $C_p(X)$. We end this section with the observation that any proper subgroup of $C(X)$ has empty interior in the point-open topology. These facts lead to a characterisation of pseudocompact spaces.

	\section{Closure of ideals in $C_p(X)$}
	
	We begin our discussions with the very obvious choice of maximal ideals of the ring $C(X)$. Since $C_p(X)$ is a topological ring, each maximal ideal in $C(X)$ is either closed or dense in the space $C_p(X)$. Now, for $p\in X$, we know that the set $M_p=\{f\in C(X)\colon f(p)=0 \}$ forms a fixed maximal ideal in $C(X)$. Moreover, any fixed maximal ideal in $C(X)$ is of the form $M_p$ for some $p\in X$. We first observe that each such ideal is closed in $C_p(X)$.
	
	\begin{theorem} \label{tClosureM_p}
		Each fixed maximal ideal in $C_p(X)$ is closed.
	\end{theorem}
	
	\begin{proof}
		Fix $p\in X$ and if possible, let there exist $f\in cl_p(M_p)\setminus M_p$. Then $f(p)\neq 0$ and there exists $g\in M_p\cap B_p(f,p,\frac{|f(p)|}{2})$. It follows that $|f(p)|<\frac{|f(p)|}{2}$, which is impossible. 
	\end{proof}
	
	Two very natural questions arise from the above theorem: \begin{enumerate}[label=\arabic*.]
		\item Are all fixed ideals in $C(X)$ closed in $C_p(X)$?
		\item Are the free maximal ideals in $C(X)$ closed or dense in $C_p(X)$?
	\end{enumerate}
	
	In reference to the second question, we recall that a free maximal ideal in $C(X)$ is of the form $M^p=\{f\in C(X)\colon p\in cl_{\beta X}Z(f) \}$, where $p\in \beta X \setminus X$. In this context, we have the following result.
	
	\begin{theorem} \label{tClosureM^p}
		Each free maximal ideal in $C_p(X)$ is dense.
	\end{theorem}
	
	\begin{proof}
		Let $p\in \beta X\setminus X$ and consider a basic open set $B(f,x_1,x_2,\cdots,x_n,\epsilon)$ in $C_p(X)$, where $f\in C(X)$, $x_1,x_2,\cdots,x_n\in X$ and $\epsilon$ is a positive real number. For each $i\in \{1,2,\cdots, n \}$, there exists $h_i\in C^*(X)$ with the property that $x_i\in \beta X\setminus cl_{\beta X}Z(h_i)\subseteq \beta X\setminus  \{p\}$. Moreover, for each $i\in \{1,2,\cdots, n \}$, there exists $t_i\in C^*(X)$ such that $t_i(x_i)=f(x_i)$ and $t_i(Z(h_i)\cup \{x_1,x_2,\cdots,x_{i-1},x_{i+1},\cdots,x_n \})=\{0\}$. Finally, define $t=\sum\limits_{i=1}^nt_i$. Then it can be easily verified that $t\in B(f,x_1,x_2,\cdots,x_n,\epsilon)\cap M^p$. Indeed, for each $i\in \{1,2,\cdots, n \}$, $h_i\in M^p$ and $z(h_i)\subseteq Z(t_i)$. Since $M^p$ is a $z$-ideal,  it follows that each $t_i\in M^p$ and consequently, $t\in M^p$. That $t\in B(f,x_1,x_2,\cdots,x_n,\epsilon)$ is immediate.
	\end{proof}
	
	******\begin{corollary}
		The number of dense maximal ideals in $C_p(X)$ is $|\beta X\setminus X|$ and that of the closed maximal ideals in $C_p(X)$ is $|X|$.
	\end{corollary}
	
	The next observation, which characterises compactness of the space $X$, follows directly from the above theorems.
	
	\begin{corollary} \label{ccompact}
		A maximal ideal, $M$ is $C(X)$ is closed if and only if it is fixed. Consequently, every maximal ideal in $C_p(X)$ is closed if and only if $X$ is compact.
	\end{corollary}
	
	We note that in proof of Theorem \ref{tClosureM^p}, the regularity of the space $\beta X$ is very subtly used to construct each $h_i$. Using the same idea, we conjecture that the ideals of the form $O^p=\{f\in C(X)\colon cl_{\beta X}Z(f)\;is\;a\;neighbourhood\;of\;p\;in\;\beta X \}$ are also dense in $C_p(X)$, for any $p\in \beta X\setminus X$.
	
	\begin{theorem} \label{tClosureO^p}
		For $p\in \beta X\setminus X$, $O^p$ is dense in $C_p(X)$.
	\end{theorem}
	
	\begin{proof}
		Let $p$ be a point in $\beta X\setminus X$, and $B(f,x_1,x_2,\cdots,x_n,\epsilon)$, a basic open set in $C_p(X)$, where $f\in C(X)$, $x_1,x_2,\cdots, x_n\in X$ and $\epsilon$ is a positive real number. Now, as $p\notin X$, $p$ has a closed neighbourhood $F$ in $\beta X$, which misses the points $x_1,x_2,\cdots,x_n$. Fix $i\in \{1,2,\cdots,n \}$. Then, there exists $h_i\in C^*(X)$ such that $x_i\in \beta X\setminus cl_{\beta X}Z(h_i)\subseteq \beta X\setminus F$. Furthermore, one can find $t_i\in C^*(X)$ such that $t_i(x_i)=f(x_i)$ and $t_i(Z(h_i)\cup \{x_1,x_2,\cdots,x_{i-1},x_{i+1},\cdots,x_n \})=\{0\}.$ Proceeding as in the proof of Theorem \ref{tClosureM^p}, as $O^p$ is a $z$-ideal in $C(X)$, it follows that $\sum\limits_{i}^nf_i\in B(f,x_1,x_2,\cdots,x_n,\epsilon)\cap O^p$.
	\end{proof}
	
	The above theorem guarantees that $cl_p(O^p)=C(X)$, whenever $p\in \beta X\setminus X$. Naturally, we wonder how the closure of $O^p=O_p$ will behave, when $p\in X$. We will come back to this soon. Before that, we make an observation that $O^p$ is a free ideal in $C(X)$, whenever $p\in \beta X\setminus X$. This ideal is shown to be dense in $C_p(X)$. This raises the whether the same conclusion can be made for any free ideal in $C(X)$. In this context, let us first recall some well-known free ideals of $C(X)$ and discuss their closure in the space $C_p(X)$.
	
	\begin{theorem} \label{tClosureC_K(X)C_psi(X)}
		The following assertions hold for a Tychonoff space $X$.
		
		\begin{enumerate}[label=\arabic*.]
			\item $cl_p(C_K(X))=\{f\in C(X)\colon f(X\setminus X_K)=\{0\} \}$, where $X_K$ denotes the set of all such points $x\in X$, which have a compact neighbourhood in $X$.
			
			\item $cl_p(C_\psi(X))=\{f\in C(X)\colon f(X\setminus X_\psi)=\{0\} \}$, where $X_\psi$ denotes the set of all such points $x\in X$, which have a pseudocompact neighbourhood in $X$. 
			
		\end{enumerate}
	\end{theorem}
	
	\begin{proof}
		We shall prove only the first assertion, as the second one can be proved analogously. Suppose $J=\{f\in C(X)\colon f(X\setminus X_K)=\{0\} \}$. Let $f\in J$ and consider an arbitrary basic open neighbourhood, $B(f,x_1,x_2,\cdots,x_n,\epsilon)$ of $f$. If each $x_i\in Z(f)$, then $\boldsymbol{0}\in B(f,x_1,x_2,\cdots,x_n,\epsilon)\cap C_K(X)$. Without loss of generality, let $x_i\in X\setminus Z(f)\subseteq X_K$ for all $i\in \{1,2,\cdots,n \}$. Fix $i\in \{1,2,\cdots,n \}$. Then $x_i$ has a compact neighbourhood $C_i$ in $X$. So, there exists $g_i\in C(X)$ such that $x_i\in X\setminus Z(g_i)\subseteq \overline{X\setminus Z(g_i)}\subseteq C_i$ and $g_i(x_i)=f(x_i)$, as $x_i\notin Z(f)$. Moreover, there exists $h_i\in C^*(X)$ such that $h_i(x_i)=1$ and $h_i(\{x_1,x_2,\cdots,x_{i-1},x_{i+1},\cdots, x_n\})=\{0\}$. Then see that $g_i\in C_K(X)$ and so $g_ih_i\in C_K(X)$, for each $i\in \{1,2,\cdots,n \}$. Then $g=\sum\limits_{i=1}^ng_ih_i\in C_K(X)$ and $g(x_i)=f(x_i)$, for each $i\in \{1,2,\cdots,n \}$. Consequently, $g\in B(f,x_1,x_2,\cdots,x_n,\epsilon)\cap C_K(X)$. 
		
		To show the reverse inclusion, it is sufficient to show that $J$ is closed in $C_p(X)$ and $C_K(X)\subseteq J$. To see the former, choose $f\in cl_pJ$ and fix $x_0\in X\setminus L$ arbitrarily. For an arbitrary $\epsilon>0$, there exists $g_\epsilon\in B(f,x_0,\epsilon)\cap J$. It follows that $g_\epsilon(x_0)=0$, for all $\epsilon>0$. Consequently, $f(x_0)=0$ for any $x_0\in X\setminus L$. Therefore, $f\in J$. Finally, let $f\in C_K(X)$ and  $f(x)\neq 0$. Then $\overline{X\setminus Z(f)}$ forms a compact neighbourhood of $x$ and so, $x\in L$. This ensures that $f\in J$.
	\end{proof}
	
	The proof of the preceding theorem can be adapted to a broader class of ideals in $C(X)$. This is established in the following result and can be proved through analogous reasoning.
	
	\begin{theorem} \label{tClosureC_P}
		For an ideal, $\mathscr{P}$ of closed subsets of $X$, $cl_p(C_\mathscr{P}(X))=\{f\in C(X)\colon f(X\setminus X_\mathscr{P})=\{0\} \}$, where $X_\mathscr{P}$ denotes the set of all such points $x\in X$, which have a neighbourhood, $P$ in $X$ such that $P\in \mathscr{P}$,  that is $X$ is locally $\mathscr{P}$ at $x$ \cite{ARN2022}.
	\end{theorem}
	
	\begin{remark}
		Theorem \ref{tClosureC_K(X)C_psi(X)} tells us that $X\setminus X_K\subseteq \bigcap Z[C_K(X)]$. Moreover, if $x\in X_K$, then one can construct a function in $C_K(X)$ which does not vanish at $x$. Therefore, $X\setminus X_K= \bigcap Z[C_K(X)]$. Similarly, we have $X\setminus X_\mathscr{P}= \bigcap Z[C_\mathscr{P}(X)]$, for any ideal $\mathscr{P}$ of subsets of $X$. 
	\end{remark}
	
	Theorem \ref{tClosureC_K(X)C_psi(X)} helps us characterise local compactness and local pseudocompactness, as can be seen in the following corollary.
	
	\begin{corollary} \label{cLocal}
		The following assertions hold for a Tychonoff space $X$.
		
		\begin{enumerate}[label=\arabic*.]
			\item $cl_p(C_K(X))$ is dense in $C_p(X)$ if and only if $X$ is locally compact.
			
			\item $cl_p(C_\psi(X))$ is dense in $C_p(X)$ if and only if $X$ is locally pseudocompact.
			
		\end{enumerate}
	\end{corollary}
	
	Moreover, we recall that $C_K(X)$  is a free ideal if and only if $X$ is locally compact and $C_\psi(X)$  is a free ideal if and only if $X$ is locally pseudocompact. The next corollaries unify these observations.
	
	\begin{corollary} \label{cLocal=FreeCompact}
		The following statements are equivalent.
		\begin{enumerate}[label=\arabic*.]
			\item $X$ is locally compact.
			\item $C_K(X)$  is a free ideal in $C(X)$.
			\item $C_K(X)$ is dense in $C_p(X)$.
		\end{enumerate}
	\end{corollary}
	
	\begin{corollary} \label{cLocal=FreePsuedoCompact}
		The following statements are equivalent.
		\begin{enumerate}[label=\arabic*.]
			\item $X$ is locally pseudocompact.
			\item $C_\psi(X)$  is a free ideal in $C(X)$.
			\item $C_\psi(X)$ is dense in $C_p(X)$.
		\end{enumerate}
	\end{corollary}
	
	Recall that for a space $X$ and an ideal $\mathscr{P}$ of closed subsets of $X$, $X$ is said to be locally $\mathscr{P}$ if each  point of $X$ has a neighborhood in $\mathscr{P}$, that is, if $X$ is locally $\mathscr{P}$ at each point in $X$ \cite{ARN2022}. Keeping this in mind, we have a more general result.
	
	\begin{corollary}
		\label{cLocallyP_free_dense}
		The following statements are equivalent.
		\begin{enumerate}[label=\arabic*.]
			\item $X$ is locally $\mathscr{P}$.
			\item $C_\mathscr{P}(X)$  is a free ideal in $C(X)$.
			\item $C_\mathscr{P}(X)$ is dense in $C_p(X)$.
		\end{enumerate}
	\end{corollary}

	The relationship between the property of being a free ideal and that of dense-ness in $C_p(X)$ as reflected by the ideals $C_K(X)$ and $C_\psi(X)$ presents us with a natural question: are all free ideals in $C(X)$ dense in $C_p(X)$? We answer this in the affirmative with the help of a much stronger result. We first establish a lemma which can be proved following steps used in the proof of  Theorem \ref{tClosureC_K(X)C_psi(X)}, to prove that the ideal $J$ is closed. 
	
	\begin{lemma} \label{lM_A}
		Let $A\subseteq X$ and $M_A=\{f\in C(X)\colon f(A)=\{0\} \}$ is a closed ideal in $C_p(X)$.
	\end{lemma}
	
	\begin{theorem} \label{tClosureAll}
		For an ideal $I$ in $C(X)$, \[cl_p(I)=\{f\in C(X)\colon \bigcap Z[I]\subseteq Z(f) \}. \]
	\end{theorem}
	
	\begin{proof}
		Denote the ideal $\{f\in C(X)\colon \bigcap Z[I]\subseteq Z(f) \}$ by $J$. It is clear that $I\subseteq J$. By Lemma \ref{lM_A}, we get the inclusion $cl_p(I)\subseteq J$. Now let $f\in J$ and consider a basic open neighbourhood $B(f,x_1,x_2,\cdots,x_n,\epsilon)$, of $f$. We have, $\bigcap Z[I]\subseteq Z(f)$. Let $k\in \{1,2,\cdots,n \}$ be such that $x_i\in Z(f)$, for all $i\in \{1,2,\cdots,k \}$ and $x_j\in X\setminus Z(f)\subseteq X\setminus \bigcap Z[I]$, for all $j\in \{k+1,k+2,\cdots,n \}$. Fix $j\in \{k+1,k+2,\cdots,n \}$. Then, as $x_j\notin \bigcap Z[I]$, there exists $g_j\in I$ such that $g_j(x_j)=f(x_j)$. Furthermore, there exists $h_j\in C(X)$ such that $h_j(x_j)=1$ and $h_j(x_i)=0$ whenever $i\in \{1,2,\cdots,n \}\setminus \{j\}$. Then $\sum\limits_{i=k+1}^ng_ih_i\in I\cap B(f,x_1,x_2,\cdots,x_n,\epsilon)$ which forces $f\in cl_p(I)$.
	\end{proof}
	
	The subsequent results are immediate from Theorem \ref{tClosureAll}.
	
	\begin{theorem} \label{cClosedIdeals}
		The family $\{M_A\colon A\text{ is a closed subset of }X \}$ constitutes the complete collection of closed ideals in $C_p(X)$.
	\end{theorem}
	
	\begin{corollary} \label{cFree=Dense}
		An ideal $I$ in $C_p(X)$ is dense if and only if it is a free ideal.
	\end{corollary}
	
	\begin{corollary} \label{cClosed=>Z}
		For an ideal $I$ in $C(X)$, $cl_p(I)$ is a $z$-ideal and hence, if an ideal $I$ in $C_p(X)$ is closed, then it is a $z$-ideal in $C(X)$.
	\end{corollary}
	\begin{corollary}
		A principal ideal $I$ is closed in $C_p(X)$ if and only if $I$ is a $z$-ideal.
	\end{corollary}
	
	We highlight at this point that principal ideals are not in general $z$-ideals. Indeed, the ideal $I=\langle i \rangle$ is not a $z$-ideal in $C_p(\mathbb{R})$, where $i\colon \mathbb{R}\longrightarrow \mathbb{R}$ is defined as $i(x)=x$, for all $x\in \mathbb{R}$ \cite[Section 2.4]{GJ1976}.
	
	Further note that Theorem \ref{tClosureM_p}, Theorem \ref{tClosureM^p} and Theorem \ref{tClosureO^p} can also be deduced using Theorem \ref{tClosureAll}. Moreover,  Theorem \ref{tClosureAll} helps us to answer the question raised before regarding the closure of the ideals $O_p$, where $p\in X$.
	
	\begin{corollary} \label{cClosureO_p}
		Suppose $p\in X$. The closure of the ideal $O_p$ in the space $C_p(X)$ is the ideal $M_p$.
	\end{corollary}

	The above corollary leads us to the following characterisation of $P$-spaces, using the space $C_p(X)$.
	
	\begin{theorem} \label{tO_pP_Space}
		$O_p$ is closed in $C_p(X)$ for all $p\in X$ if and only if $X$ is a $P$-space.
	\end{theorem}
	
	\begin{proof}
		If each $O_p$ is closed in $C_p(X)$, where $p\in X$; then it follows from Corollary \ref{cClosureO_p} that $O_p=M_p$. By \cite[Exercise 4J]{GJ1976}, we conclude that $X$ is a $P$-space. Converse can also be achieved using \cite[Exercise 4J]{GJ1976} and Corollary \ref{cClosureO_p}.
	\end{proof}
	
	Recall that a prime ideal $P$ in $C(X)$ lies between the ideals $O^p$ and $M^p$, for some $p\in \beta X$ \cite[Theorem 7.15]{GJ1976}. Building from this and the above results, we have the following observation.
	
	\begin{corollary} \label{cPrimeClosure}
		Let $P$ be a prime ideal in $C(X)$, then
		\begin{enumerate} [label=(\roman*)]
			\item $cl_p(P)=C(X)$, if $P$ is a free ideal, and
			\item $cl_p(P)=M^p$, if $P$ is fixed.
		\end{enumerate} 
	\end{corollary}

	We shift our focus to $z^\circ$-ideals. 
	
	\begin{remark} \label{remz0}
		We realise there is no general implication between $z^\circ$-ideals in $C(X)$ and closed ideals in $C_p(X)$. Note that if $X$ has a point $p$ which is not an almost $P$-point, then $M_p$ is closed in $C_p(X)$, but is not a $z^\circ$-ideal. Indeed, as $p\in X$ is not an almost $P$-point, there exists $f\in C(X)$ such that $p\in Z(f)$ but $int\;Z(f)=\emptyset$. Now, $int\;Z(f)=int\;Z(\boldsymbol{1})$ with $f\in M_p$ but $\boldsymbol{1}\notin M_p$. That $M_p$ is closed has been noted in Theorem \ref{tClosureM_p}. Moreover, for any $p\in X$, $O_p$ is a $z^\circ$-ideal which is not closed, in general (Corollary \ref{cClosureO_p}).
	\end{remark}
	
	We use the preceding remark to formulate a necessary and sufficient condition for $X$ to be an almost $P$-space.
	
	\begin{theorem} \label{tAlmostPSpace}
		Every closed ideal in $C_p(X)$ is a $z^\circ$-ideal if and only if $X$ is an almost $P$-space.
	\end{theorem}
	
	\begin{proof}
		Suppose $p\in X$ is not an almost $P$-point. Then, as seen in Remark \ref{remz0}, the ideal $M_p$ is closed in $C_p(X)$, but is not a $z^\circ$-ideal in $C(X)$. Conversely let $X$ be an almost $P$-space and $I$, a closed ideal in $C(X)$. Then by Corollary \ref{cClosed=>Z} and \cite[Theorem 2.14]{AKA1999}, $I$ is a $z^\circ$-ideal.
	\end{proof}
	
	Combining Theorem 2.14 in \cite{AKA1999},  Remark \ref{remz0} and Theorem \ref{tAlmostPSpace}, we have the following result.
	
	\begin{theorem} 
		The following statements are equivalent. 
		
		\begin{enumerate}[label=\arabic*]
			\item $X$ is an almost $P$-space.
			\item Every $z$-ideal in $C(X)$ is a $z^\circ$-ideal.
			\item Every maximal ideal in $C(X)$ is a $z^\circ$-ideal.
			\item Every fixed maximal ideal in $C(X)$ is a $z^\circ$-ideal.
			\item Every closed ideal in $C_p(X)$ is a $z^\circ$-ideal.
		\end{enumerate}
	\end{theorem}
	
	We now proceed to examine the behaviour of essential ideals with respect to the space $C_p(X)$. Note that in $C(\mathbb{R})$, both $M_0$ and $O_0$ are essential ideals. However, $M_0$ is closed in $C_p(X)$ and $O_0$ is neither closed, nor dense in $C_p(X)$. By Corollary \ref{cFree=Dense}, the only dense essential ideals are the ones that are free. This fact, together with \cite[Corollary 3.5]{A1995} yields the next result.
	
	\begin{theorem}
		The following statements are equivalent.
		\begin{enumerate}[label=\arabic*]
			\item Every essential ideal in $C_p(X)$ is dense.
			\item Every essential ideal in $C(X)$ is free.
			\item $X$ is discrete.
		\end{enumerate}
	\end{theorem}
	
	We further wonder for what kind of topological spaces are all the essential ideals in $C(X)$ closed in $C_p(X)$. Before getting into that, let us recall the definition of convex ideal and absolutely convex ideal in $C(X)$.
	\begin{definitions}
		\	\begin{enumerate}
			\item An ideal $I$ in $C(X)$ is said to be a convex ideal if whenever $\textbf{0}\leq f\leq g$, where $g\in  I$ and $f\in C(X)$, then $f\in I$.
			\item An ideal $I$ in $C(X)$ is said to be an absolutely convex ideal if when ever $|f|\leq |g|$, where $g\in  I$ and $f\in C(X)$, then $f\in I$.
		\end{enumerate}
	\end{definitions}
	We know that every $z$-ideal $I$ in $C(X)$ is absolutely convex and every absolutely convex ideal is convex \cite[Chapter 5]{GJ1976}. Therefore, in view of Theorem \ref{cClosed=>Z} we have the following theorem.
	\begin{theorem}
		Every closed ideal $I$ in $C_p(X)$ is absolutely convex and hence convex.
	\end{theorem}
	
	We now that the only situation in which every ideal (resp.\ $z$-ideal, $z^\circ$-ideal, etc) is closed in the space $X_p(X)$ is when $X$ is a finite set.
	
	\begin{theorem}
		For a space $X$, the following statements are equivalent:
		\begin{enumerate}
			\item Every ideal is closed in $C_p(X)$.
			\	\item  Every convex ideal is closed in $C_p(X)$.
			\item Every absolutely  convex ideal is closed in $C_p(X)$.
			\item Every $z$-ideal is closed in $C_p(X)$.
			\item Every $z^\circ$-ideal is closed in  $C_p(X)$.
			\item  Every essential ideal is closed in $C_p(X)$.
			\item Every prime ideal is closed in $C_p(X)$.
			\item  $X$ is finite.

		\end{enumerate}
	\end{theorem}
	\begin{proof}
		$(1) \implies (2) \implies (3)\implies (4)\implies (5)$: Trivial.
		
		$(1)\implies (6)/(7)$: Trivial.
		
		$(5)/(6)\implies (8)$: It is enough to show that $X$ is a compact $P$-space \cite[4K]{GJ1976}. If possible, let $X$ be not compact. Choose, a point $p\in \beta X\setminus X$. Then by Theorem \ref{tClosureO^p}, $O^p$ is a dense $z^\circ$-ideal/essential ideal, which is a contradiction. Hence, $X$ is compact. Now, it remains to show that $X$ is a $P$-space. Let $p\in X$. Then by Theorem \ref{cClosureO_p}, $O_p=cl_p(O_p)=M_p$, for all $p\in X$. Consequently, $X$ is a $P$-space, \cite[4J]{GJ1976}.
		
		$(7)\implies (8)$: Since every maximal ideal is a prime ideal then by Theorem \ref{ccompact}, $X$ is compact. Let $P$ be a prime ideal in $C(X)$. Then $O_p\subseteq P\subseteq M_p$ for some $p\in X$. Therefore, $P=cl_p(P)=M_p$ and hence $P$ is a maximal ideal. Consequently, $X$ is a $P$-space \cite[4J]{GJ1976}.
		
		$(8)\implies (1)$: Since $X$ is finite, then $C_p(X)=C_m(X)$. Consequently, every ideal is closed in $C_p(X)$ as $X$ is also a $P$-space.
	\end{proof}

	It is an appropriate time to recall that following results.
	\begin{theorem}\cite[Exercise 7Q]{GJ1976}
		\begin{enumerate}
			\item An ideal $I$ in $C(X)$, $cl_m(I)= \bigcap\{M^p: p\in \beta X\text{ and }I\subseteq M^p\}$.
			\item  Every ideal $I$ is closed in $C_m(X)$ if and only if $X$ is a $P$-space.
		\end{enumerate}
	\end{theorem}
	The above theorem leads us to the following question.
	\begin{question}
		For an ideal $I$ in $C(X)$, can we characterize $cl_p(I)$ in terms of the maximal ideals on $C(X)$?
	\end{question}

	The next theorem gives an affirmative answer to the above question.
	\begin{theorem}\label{maximalClosureAll}
		For an ideal $I$ in $C(X)$, $cl_p(I)=\bigcap\{cl_p(M^p):p\in \beta X\text{ and }I\subseteq M^p\}.$
	\end{theorem}
	\begin{proof}
		If $I$ is free then nothing to proof.
		
		Let $I$ be fixed. Also assume that $P=\bigcap\{cl_p(M^p):p\in \beta X\text{ and }I\subseteq M^p\}$. Then it is easy to check that $cl_p(I)\subseteq P$. Let us choose, $g\in P$ and $c\in \bigcap Z[I]$. Hence, $I\subseteq M_c$ and therefore, $g\in cl_p(M_c)=M_c$. Consequently, we have $g(c)=0$, which implies that $g\in cl_p(I)$.
		
	\end{proof}
	In view of Theorem \ref{tClosureAll} and Theorem \ref{maximalClosureAll}, we have the following remark.
	\begin{remark}
		For an ideal $I$ in $C(X)$, we observe that	$$cl_p(I)=\{f\in C(X)\colon \bigcap Z[I]\subseteq Z(f) \}=\bigcap\{cl_p(M^p):p\in \beta X\text{ and }I\subseteq M^p\}=M_A$$ where $M_A=\begin{cases}
			\bigcap\limits_{x\in \bigcap Z[I]}M_x,\text{ when } I\text{ is fixed}\\
			C(X), \text{ when } I \text{ is free} 
		\end{cases}$.
	\end{remark}

	
	In light of Theorem \ref{tClosureAll}, we now realise the structure of $\bigcap Z[I+J]$, for any two ideals $I$ and $J$ in $C(X)$. It is routine to verify that for any two ideals $I$ and $J$ in $C(X)$, \[\bigcap Z[I+J]=\bigcap Z[I] \cap \bigcap Z[J]. \] Our focus is to deduce the relationship between $cl_p(I+J)$ and $cl_p(I)+cl_p(J)$. Since, for an ideal $I$ in $C(X)$, $cl_p(I)=M_{\bigcap Z[I]}$, we begin with the discussions involving the closure of sum of two closed ideals in $C_p(X)$. We raise the following question in this regard.
	
	\begin{question}
		If $I$ and $J$ are closed ideals in $C_p(X)$, then is $I+J$ also a closed ideal in $C_p(X)$?
	\end{question}
	
	Note that, for a pair of ideals $I,J$ in $C(X)$, $cl_p(I+J)=\{f\in C(X)\colon \bigcap Z[I] \cap \bigcap Z[J]\subseteq Z(f)  \} $. For a space $X$, we reserve the notation $\mathcal{I}$ to denote the family of all such closed ideals in $C_p(X)$ such that for any pair of distinct members $I,J\in \mathcal{I}$, $\bigcap Z[I]\cap \bigcap Z[J]=\emptyset$. So, for a pair of ideals $I,J$ in $\mathcal{I}$, $cl_p(I+J)=C(X)$. In this regard, we are able to characterise normal spaces using the sum of two ideals in $\mathcal{I}$.
	
	\begin{theorem} \label{tNormal_sum}
		A space $X$ is normal if and only if for any pair of ideals $I,J$ in $\mathcal{I}$, $I+J=C(X)$.
	\end{theorem}
	
	\begin{proof}
		Suppose $X$ is normal and $I,J\in \mathcal{I}$.  Since $I$ and $J$ are closed ideals, there exist closed sets $A,B$ in $X$ such that $I=M_A$ and $J=M_B$ with $A\cap B=\emptyset$. By normality of $X$, there exists $f\in C(X)$ such that $f(A)=\{0\}$ and $f(B)=\{1\}$. Again, $Z(f)$ and $B$ are disjoint closed sets. So, there exists $g\in C(X)$ such that $g(Z(f))=\{1\}$ and $g(B)=\{0\}$. It follows that $f\in M_A$ and $g\in M_B$ with $Z(f)\cap Z(g)=\emptyset$. So, $f^2+g^2$ is a unit in $M_A+M_B$. It follows that $I+J=C(X)$.
		
		Conversely, suppose $X$ is not a normal space. Then there exists closed sets $A,B$ in $X$ which are not completely separated. This ensures that $\boldsymbol{1}\notin M_A+M_B$. Indeed, if $\boldsymbol{1}=f+g$ with $f\in M_A$ and $g\in M_B$, then $f(A)=\{0\}$ and $f(B)=\{1\}$, which ensures that $A$ and $B$ are completely separated sets. Therefore, $I+J\subsetneqq C(X)$.
	\end{proof}
	
	The following result is a direct consequence of the above theorem.
	
	\begin{corollary}
		A space $X$ is normal if and only if for any pair of ideals $I,J$ in $\mathcal{I}$, $I+J$ is closed in $C_p(X)$.
	\end{corollary}
	
	The question remains: what can be said about the closed ideals of $C_p(X)$ that do not belong to $\mathcal{I}$? We answer this question for a specific class of topological spaces.
	
	\begin{theorem} \label{tPerfectly_normal}
		Suppose $X$ is a perfectly normal space. For any pair, $I,J$ of closed ideals in $C_p(X)$, $I+J$  is closed in $C_p(X)$.
	\end{theorem}
	
	\begin{proof}
		Suppose $I$ and $J$ are closed ideals in $C_p(X)$. Then there exist closed subsets, $A,B$ of $X$ such that $I=M_A$ and $J=M_B$. Since $X$ is perfectly normal, there exist $f,g\in C(X)$ such that $A=Z(f)$ and $B=Z(g)$. It then follows that $M_A+M_B=M_{Z(f^2+g^2)}$, which is closed in $C_p(X)$.
	\end{proof}
	
	The converse of the above theorem remains an unanswered question.

	\section{Closure of subrings and some other special subsets of $C_p(X)$}
	
	Denote the set of all units in $C(X)$ by $U(X)$ and the set of all divisors of zero in $C(X)$, along with $\boldsymbol{0}$, by $D(X)$. We begin this section by determining the closures of these two subsets of $C_p(X)$. We first realise that the set $D(X)$ is either closed or dense in $C_p(X)$.
	
	\begin{theorem} \label{tD(X)Closure}
		The closure of $D(X)$, in the space $C_p(X)$ is given by: \[cl_p(D(X))=\begin{cases}
			C_p(X), &when\;X\;is\;infinite \\
			D(X), &when\;X\;is\;finite
		\end{cases}. \] Consequently, $D(X)$ is either closed or dense in $C_p(X)$.
	\end{theorem}
	
	\begin{proof}
		First suppose that $X$ is an infinite set. Consider $f\in C(X)\setminus D(X)$ and a basic open neighborhood $B(f,F,\epsilon)$ of $f$ in $C_p(X)$, where $F$ is a finite subset of $X$ and $\epsilon$ is a positive real number. If $f(F)=\{0\}$, then $\boldsymbol{0}\in B(f,F,\epsilon)\cap D(X)$. Suppose that there exists $x_0\in F$ with $f(x_0)\neq 0$. Note that as $f\notin D(X)$, $\overline{X\setminus D_f}=X$. This ensures that there exists $y\in X\setminus (Z(f)\cup F)$. By regularity of $X$, there exists an open set $U$ in $X$ such that $y\in U\subseteq \overline{U}\subseteq X\setminus (Z(f)\cup F)$. Again, there exists $g\in C(X)$ satisfying $g(\overline{U})=\{0\}$ and $g(x)=f(x)$, for all $x\in F$. It follows that $g\in B(f,F,\epsilon)\cap D(X)$. This ensures that $D(X)$ is dense in $C_p(X)$.
		
		Now, if $X$ is finite, then $C_p(X)$ is homeomorphic to $\mathbb{R}^n$, endowed with the Euclidean topology, where $n=|X|$. Also,  $D(X)=\{ (x_1,x_2,\cdots, x_n)\colon x_1x_2\cdots x_n=0 \}$, which is closed in $\mathbb{R}^n$.
	\end{proof}
	
	The following corollary is an immediate consequence to the above result.
	
	\begin{corollary}
	\	\begin{enumerate}[label=\arabic*.]
		\item $D(X)$ is dense in $C_p(X)$ if and only if $X$ is infinite.
		\item  $D(X)$ is closed in $C_p(X)$ if and only if $X$ is finite.
	\end{enumerate}
	\end{corollary}
	
	
	We are now concerned about the closure of $U(X)$ in $C_p(X)$. Consider $X$ to be the real line and $f\colon \mathbb{R}\longrightarrow \mathbb{R}$ be a function in $C(\mathbb{R})$ such that there exists $x,y\in \mathbb{R}$ with $f(x)<0$ and $f(y)>0$. Then, for any $\delta$ satisfying $0<\delta<\frac{1}{2}\min \{f(x),f(y) \}$, we have $B(f,x,y,\delta)\cap U(X)=\emptyset$. The significant factor here is that $\mathbb{R}$ is a connected topological space. Moreover, if $X$ is a topological space having a component $C$ with the property that $|C|\geq 2$ (that is, if $X$ is not totally disconnected), then similar arguments can be used to ensure that $U(X)$ is not dense in $C_p(X)$. However, the converse of this observation is not true. On this note, we recall that for a point $x\in X$, the \textit{quasicomponent} of $x$ is defined as the intersection of all clopen neighbourhoods of $x$ in $X$. Furthermore, a  topological space $X$ is said to be \textit{totally separated} if the quasicomponents in $X$ are precisely the singleton sets. We establish a connection between the denseness of $U(X)$ in $C_p(X)$ and the property of $X$ being totally separated.
	
	\begin{theorem} \label{tTotallySeparated}
		$U(X)$ is dense in $C_p(X)$ if and only if $X$ is totally separated.
	\end{theorem}
	
	\begin{proof}
		Let $X$ be totally separated and consider a basic open set $B(f,x_1,x_2,\cdots,x_n,\epsilon)$ in $C_p(X)$. Then, there exist $n$ pairwise disjoint clopen subsets, $U_1,U_2,\cdots,U_n$ of $X$ such that $x_i\in U_i$ for all $i\in \{1,2,\cdots,n\}$. For each  $i\in \{1,2,\cdots,n\}$, choose $r_i\in (f(x_i)-\epsilon,f(x_i)+\epsilon)\setminus \{0\}$. Define $g\colon X\longrightarrow \mathbb{R}$ as $g(x)=\begin{cases}
			r_i &when\;x\in U_i\;i\in \{1,2,\cdots,n\} \\
			1 &when\;x\in X\setminus \bigcup\limits_{i=1}^nU_i
		\end{cases}.$ See that $g\in B(f,x_1,x_2,\cdots,x_n,\epsilon)\cap U(X)$.
		
		Conversely, suppose there exist a pair of distinct points $x,y$ in $X$ with the property that they lie in the same quasicomponent. Now, there exists $f\colon X\longrightarrow [-1,1]$ such that $f(x)=-1$ and $f(y)=1$. Then $B(f,x,y,\frac{1}{2})\cap U(X)=\emptyset$. Indeed if there exists a unit $u\in B(f,x,y,\frac{1}{2})$, then $u(x)<0$ and $u(y)>0$. It follows that the set $\{t\in X\colon f(t)>0 \}=\{t\in X\colon f(t)\geq 0 \}$ is a clopen set containing $y$ but missing $x$.
	\end{proof}
	
	The following corollary is an immediate consequence of the facts that every zero-dimensional space is totally separated and every extremally disconnected space is totally separated. Although this is a known fact, we shall prove it as a lemma. This will provide us with abundance of examples of spaces for which $U(X)$ shall be dense in $C_p(X)$. 
	
	\begin{lemma}
	\	\begin{enumerate}[label=\arabic*.]
		\item A zero-dimensional space is totally separated.
		\item An extremally disconnected space is totally separated.
	\end{enumerate}
	\end{lemma}
	
	\begin{proof}
	\	\begin{enumerate}[label=\arabic*.]
		\item Suppose X is a zero-dimensional space and $x\in X$. If possible, let the quasicomponent $Q_x$ of $x$ contains $y\neq x$. Then, by zero-dimensionality, there exists pair of disjoint clopen sets $B_1,B_2$ such that $x\in B_1$ and $y\in B_2$. This contradicts that $y\in Q_x$.

		\item Let $X$ be an extremally disconnected space. If possible, let there exist a quasicomponent $Q$ in $X$ containing two distinct points $x$ and $y$. Then, there exist a pair of disjoint open sets $U$ and $V$ in $X$, containing $x$ and $y$ respectively. Since $X$ is an extremally disconnected space, $\overline{U}$ is open, and hence clopen, in $X$. Therefore, $y\in \overline{U}$. But $V$ is an open neighbourhood of $y$ with $U\cap V=\emptyset$, which is a contradiction.
	\end{enumerate}
	\end{proof}
	
	\begin{corollary}
		If $X$ is zero-dimensional (or, extremally disconnected), then $U(X)$ is dense in $C_p(X)$.
	\end{corollary}

	One may note from the proof of Theorem \ref{tTotallySeparated} that if a quasicomponent of $X$ has two distinct points $x,y$, any function $f\in C(X)$ with $f(x)<0$ and $f(y)>0$ fails to be in the closure of $U(X)$ in $C_p(X)$. Building from this, we present our next observation regarding the closure of $U(X)$ in $C_p(X)$. We describe a few notations, that shall be prevalent throughout the article.
	\begin{notations}
		\		\begin{itemize}
			\item $\mathcal{Q}$ denotes the collection of all quasicomponents of $X$.
			\item For each $A\subseteq X$, $A^+=\{f\in C(X)\colon f|_A\geq \boldsymbol{0}\;on\;X \}$ and $A^-=\{f\in C(X)\colon f|_A\leq \boldsymbol{0}\;on\;X \}$.
		\end{itemize}
	\end{notations}
	
	\begin{theorem} \label{tclosureUX}
		The closure of the collection of all units, $U(X)$ in $C(X)$  is given by \[cl_p(U(X))=\bigcap\limits_{Q\in \mathcal{Q}}(Q^+\cup Q^-).\] 
	\end{theorem}
	
	\begin{proof}
		Let $f\in \bigcap\limits_{Q\in \mathcal{Q}}(Q^+\cup Q^-)$. Consider a basic open neighbourhood $B(f,x_1,x_2,\cdots,x_n,\epsilon)$ of $f$ in $C_p(X)$. Suppose $Q_1,Q_2,\cdots,Q_k$ is the collection of distinct quasicomponents in $X$ such that each $Q_i$ contains at least one point of $\{x_1,x_2,\cdots,x_n\}$. It follows that there exists a collection of pairwise disjoint clopen sets $U_1,U_2,\cdots,U_k$ in $X$ such that $U_i\supseteq Q_i$, for all $i\in \{1,2,\cdots,k \}$. Define $u\colon X\longrightarrow \mathbb{R}$ as $$u(x)=\begin{cases}
			|f(x)|+\frac{\epsilon}{2}, &if\;x\in U_i\;and\;f|_{Q_i}\geq \boldsymbol{0}\;on\;X \\
			-|f(x)|-\frac{\epsilon}{2}, &if\;x\in U_i\;and\;f|_{Q_i}\leq \boldsymbol{0}\;on\;X \\
			1, &if\;x\in X\setminus \bigcup\limits_{i=1}^kU_i
		\end{cases}.$$ It is evident that $u\in C(X)$. Moreover, $u\in U(X)$. Also, $|f(x_i)-u(x_i)|=\epsilon/2$ for each $i\in \{1,2,\cdots,k\}$ and so $u\in U(X)\cap B(f,x_1,x_2,\cdots,x_n,\epsilon)$. It follows that $u\in cl_p(U(X))$. 
		
		To establish the reverse inclusion, it suffices to prove the following two statements:
		\begin{enumerate}[label=\arabic*.]
			\item $U(X)\subseteq Q^+\cup Q^-$, for any quasicomponent $Q\in \mathcal{Q}$, and
			\item For each quasicomponent $Q\in \mathcal{Q}$, $Q^+\cup Q^-$ is a closed subset of $C_p(X)$.
		\end{enumerate}  First suppose $u\in U(X)$ and $Q\in \mathcal{Q}$. If possible, let there exists $x,y\in Q$ such that $u(x)<0<u(y)$. Then the sets $V_x=\{t\in X\colon u(t)<0 \}=\{t\in X\colon u(t)\leq 0 \}$ and $V_y=\{t\in X\colon u(t)>0 \}=\{t\in X\colon u(t)\geq 0 \}$ are disjoint clopen neighbourhoods of $x$ and $y$ respectively. This contradicts that $x$ and $y$ lie in the same quasicomponent. It therefore ensures that $U(X)\subseteq Q^+\cup Q^-$, for any quasicomponent $Q\in \mathcal{Q}$. Finally, suppose $f\in C(X)$ and $Q\in \mathcal{Q}$ is such that there exists $x,y\in Q$ with $f(x)<0<f(y)$. Then it follows that $B(f,x,y,\delta)\cap (Q^+\cup Q^-)=\emptyset$, where $0<\delta<\frac{1}{2}\min \{|f(x)|,|f(y)| \}$. This completes the proof.
	\end{proof}
	
	It is now evident that if $X$ contains no non-trivial proper clopen sets, then $cl_p(U(X))$ is the collection of all non-negative and non-positive functions in $C(X)$. This phenomenon clearly characterises these kind of spaces.
	
	\begin{theorem}
		$cl_p(U(X))$ consists precisely of non-negative and non-positive functions in $C(X)$, i.e., \[cl_p(U(X))=\{f\in C(X)\colon f\geq \boldsymbol{0}\;on\;X\;or\;f\leq \boldsymbol{0}\;on\;X \} \] if and only if $X$ has no non-trivial proper clopen set.
	\end{theorem}
	
	The subsequent corollary is a direct consequence.
	
	\begin{corollary}
		If $X$ is a connected topological space, then $cl_p(U(X))$ is the collection of all non-negative and non-positive functions in $C(X)$.
	\end{corollary}
	
	In this context, we recall that every unit in $C(X)$ is a Von-Neumann regular element. So, $U(X)\subseteq V(X)$, where $V(X)$ denotes the set of all Von-Neumann regular elements in $C(X)$. Moreover, for any $f\in V(X)$, $Z(f)$ is a clopen subset of $X$. These observations naturally lead to the following theorem. 
	
	\begin{theorem} \label{tVonNeumannRegularClosure}
		$cl_p(V(X))=cl_p(U(X))=\bigcap\limits_{Q\in \mathcal{Q}}(Q^+\cup Q^-).$
	\end{theorem}
	
	\begin{proof}
		That $cl_p(V(X))\supseteq cl_p(U(X))$ is evident from the preceding discussions. To establish that $cl_p(V(X))\subseteq cl_p(U(X))$, it is sufficient to show that $V(X)\subseteq \bigcap\limits_{Q\in \mathcal{Q}}(Q^+\cup Q^-).$ The rest would follow from Theorem \ref{tclosureUX}.	Consider $f\in V(X)$ and $Q\in \mathcal{Q}$. If possible, let $f(x)<0<f(y)$ for some $x,y\in X$. Since $Z(f)$ is clopen, the sets $V_x=\{t\in X\colon f(t)<0 \}=\{t\in X\colon f(t)\leq 0 \}\setminus Z(f)$ and $V_y=\{t\in X\colon f(t)>0 \}=\{t\in X\colon f(t)\geq 0 \}\setminus Z(f)$ are disjoint clopen neighbourhoods of $x$ and $y$ respectively. This contradicts the fact that $x$ and $y$ lie in the same quasicomponent.
	\end{proof}

	We now proceed to discuss the closure of certain specific type of subrings in $C(X)$. In order to do this efficiently, we recall a few notations. For each $f\in C(X)$ and $\epsilon>0$, the set $A_\epsilon(f)$ is defined as $A_\epsilon(f)=\{x\in X\colon |f(x)|\geq \epsilon\}$, for each $\epsilon>0$.   Also, the collection $C_\infty^\mathscr{P}(X)=\{f\in C(X)\colon A_\epsilon \in \mathscr{P} \}$ forms a subring of $C(X)$ containing the ideal $C_\mathscr{P}(X)$ \cite{ARN2022}. We assert that the closure of $C_\infty^\mathscr{P}(X)$ coincides with that of $C_\mathscr{P}(X)$.
	
	\begin{theorem} \label{tClosureCInfinity}
		$cl(C_\infty^\mathscr{P}(X))=cl(C_\mathscr{P}(X))$.
	\end{theorem}
	
	\begin{proof}
		In light of Theorem \ref{tClosureC_P} and using the fact that $C_\mathscr{P}(X)\subseteq C_\infty^\mathscr{P}(X)$, it is sufficient to prove that for any $f\in C_\infty^\mathscr{P}(X)$, $f( X\setminus X_\mathscr{P})=\{0\}$. Suppose $f\in C_\infty^\mathscr{P}(X)$ and if possible, let there exist $x\in X\setminus X_\mathscr{P}$ such that $f(x)\neq 0$. Then for $\epsilon=\frac{|f(X)|}{2}$, $A_\epsilon(f)$ is a neighborhood of $x$ and $A_\epsilon(f)\in \mathscr{P}$, which contradicts $x\notin X_\mathscr{P}$.
	\end{proof}
	
	The following observations are   immediate using Theorem \ref{tClosureC_P} and Theorem \ref{tClosureCInfinity}.
	
	\begin{corollary} \label{cLocallyP}
		The following statements are equivalent.
		\begin{enumerate}[label=\arabic*.]
			\item $X$ is locally $\mathscr{P}$.
			\item $C_\mathscr{P}(X)$  is a free ideal in $C(X)$.
			\item $C_\mathscr{P}(X)$ is dense in $C_p(X)$.
			\item $C_\infty^\mathscr{P}(X)$ is dense in $C_p(X)$.
		\end{enumerate}
	\end{corollary}

	We now revisit the following result, which can be deduced from \cite[Theorem 2.1]{Mandelker1971}.
	
	\begin{theorem}
		$C_\psi(X)$ is the largest ideal of $C(X)$, contained in $C^*(X)$ and $C_\psi(X)=\{f\in C(X)\colon \forall\; g\in C(X),\;fg\in C^*(X) \}$. 
	\end{theorem}
	
	This motivates us to the following result.
	
	\begin{theorem}
	Suppose $\mathscr{P}$ is an ideal of closed sets in $X$ and define $J_\mathscr{P}=\{f\in C(X)\colon \forall\; g\in C(X),\;fg\in C_\infty^\mathscr{P}(X) \}$. Then the following assertions hold.
	\begin{enumerate}[label=\arabic*.]
		\item \label{t3.14(1)} $C_\mathscr{P}(X)\subseteq J_\mathscr{P}\subseteq C_\infty^\mathscr{P}(X)$.
			
		\item If $I$ is an ideal of $C(X)$, contained in $C_\infty^\mathscr{P}(X)$, then $I\subseteq J_\mathscr{P}$. In other words, $J_\mathscr{P}$ is the largest ideal contained in $C_\infty^\mathscr{P}(X)$.
			
		\item $cl(C_\infty^\mathscr{P}(X))=cl_p(J_\mathscr{P})=cl(C_\mathscr{P}(X))$.
		\end{enumerate}
	\end{theorem}

	\begin{proof}
	\	\begin{enumerate}[label=\arabic*.]
			\item This follows from the definition of an ideal of closed subsets in $X$ and the fact that $\boldsymbol{1}\in C(X)$.
			
			\item Let $I$ be an ideal of $C(X)$, contained in $C_\infty^\mathscr{P}(X)$. Then for any $f\in I$ and $g\in C(X)$, $fg\in I\subseteq C_\infty^\mathscr{P}(X)$ and so, $f\in J_\mathscr{P}$.
			
			\item This is evident from \ref{t3.14(1)} and Theorem \ref{tClosureCInfinity}.
		\end{enumerate}
	\end{proof}

	\begin{definition}
		A family $F\subseteq C(X)$ is said to separate points if for any two points $x,y\in X$ with $x\neq y$ there exists a $f\in X$ such that $f(x)\neq f(y)$.
	\end{definition}
	
	\begin{lemma}\label{separating Lemma}
		Let $S$ be a subring of $C(X)$ which contains all the constant functions and separates points in $X$. Then for any two points $x,y\in X$ with $x\neq y$ and $r\in \mathbb{R}$ there exists a $f\in S$ such that $f(x)=0$ and $f(y)=r$.
	\end{lemma}
	\begin{proof}
		Since $S$ separates points in $X$, then there exists a $g\in S$ such that $g(x)\neq g(y)$. Set, $h=g-g(x)$ and hence, $f=\frac{r}{h(y)}h$ is the required function.
	\end{proof}
	\begin{theorem}\label{separating Theorem}
		Let $S$ be a subring of $C(X)$ which contains all the constant functions and separates points in $X$. Then for given distinct points $x_1,x_2,\cdots,x_n\in X$, $i\in \{x_1,x_2,\cdots,x_n\}$, and $r\in \mathbb{R}$ there exists a $f_i\in S$ such that $f_i(x_i)=r$ and $f_i(x_j)=0$ for $j\neq i$.
	\end{theorem}
	\begin{proof}
		It follows from Lemma \ref{separating Lemma} that there exists a $g\in S$ such that $g(x_i)=r$. Choose, a $j\in \{x_1,x_2,\cdots,x_{i-1},x_{i+1},\cdots,x_n\}$, and hence, again by Lemma \ref{separating Lemma}, we can find a $h_j\in S$ such that $h_j(x_i)=1$ and $h_j(x_j)=0$. Therefore, $f_i=gh_1h_2\cdots h_{i-1}h_{i+1}\cdots h_n$ is our required function.
	\end{proof}
	\begin{theorem}
		Let $S$ be a subring of $C(X)$ which contains all the constant functions. Then $cl_p(S)=C(X)$ if and only if $S$ separates points in $X$.
	\end{theorem}
	
	\begin{proof}
		Let $S$ separates points in $X$. Let $f\in C(X)$ and $B(f,x_1,x_2,\cdots,x_n,\epsilon)$ be a basic open set containing $f$ in $C_p(X)$ where $x_i\in X$ for all $i=1,2,\cdots,n$ and $\epsilon>0$. Then by Theorem \ref{separating Theorem}, for each $i\in \{1,2,\cdots,i-1,i+1,\cdots,n\}$ there exists a $h_i\in S$ such that $h_i(x_i)=f(x_i)$ and $h_i(x_j)=0$ for each $j\in \{1,2,\cdots,i-1,i+1,\cdots,n\}$. It is straightforward to verify that $\sum_{i=1}^{n}h_i\in S\cap B(f,x_1,x_2,\cdots,x_n,\epsilon)$ and hence, $S$ is dense in $C_p(X)$. Conversely, let $cl_p(S)=C_p(X)$. Suppose, $x,y\in X$ with $x\neq y$ and hence, there is a $f\in C(X)$ such that $f(x)=1$ and $f(y)=2$. Choose a $g\in S\cap B(f,x,y,\frac{1}{2})$ and therefore, $g(x)\neq g(y$.) 
	\end{proof}
	
	\begin{corollary}
		Any intermediate ring $A(X)$ is dense in $C_p(X)$.
	\end{corollary}
	\begin{corollary}
		For a space $X$, the following are equivalent:
		\begin{enumerate}
			\item $X$ is pseudocompact.
			\item $C^*(X)$ is closed in $C(X)$.
			\item Any intermediate ring $A(X)$ is closed in $C_p(X)$.
		\end{enumerate}
	\end{corollary}
	
	\begin{theorem}
		Let $G$ be a proper additive subgroup of $C(X)$. Then $int_p(G)=\emptyset$.
	\end{theorem}
	\begin{proof}
		If possible, let $int_p(G)\neq \emptyset$ and choose a $f\in  int_p(G)$. Then there exist a basic open set $B(f,x_1,x_2,\cdots,x_n)$ in $C_p(X)$ containing $f$ such that $B(f,x_1,x_2,\cdots,x_n,\epsilon)\subseteq G$ where $x_1,x_2,\cdots,x_n\in X$ and $\epsilon>0$. We also choose a $g\in C(X)\setminus G$. We can find a $n\in \mathbb{N}$ such that $\frac{1}{n}\max\{|g(x_1)|,|g(x_2)|,\cdots,|g(x_n)|\}<\epsilon$. Therefore, $f+\frac{1}{n}g\in B(f,x_1,x_2,\cdots,x_n)\subseteq G$ and hence, $\frac{1}{n}g\in G$, which implies that $g\in G$- a contradiction.
	\end{proof}
	\begin{corollary}
		Let $G$ be an additive subgroup of $C(X)$ such that $int_p(G)\neq \emptyset$ then $G=C(X)$.
	\end{corollary}
	\begin{corollary}
		Let $S$ be a subring of $C(X)$, then $int_p(S)$ is either $\emptyset$ or $C(X)$.
	\end{corollary}
	\begin{corollary}
		For a space $X$, the following statements are equivalent:
		\begin{enumerate}
			\item $X$ is pseudocompact.
			\item $C^*(X)$ is open in $C_p(X)$.
			\item Any intermediate ring $A(X)$ is open in $C_p(X)$.
			\item $C^*(X)$ is closed in $C(X)$.
			\item Any intermediate ring $A(X)$ is closed in $C_p(X)$.
		\end{enumerate}
	\end{corollary}

\end{document}